\documentclass[12pt]{article}
\usepackage[english]{babel}
\usepackage{amsmath,amssymb,amsthm,url}
\usepackage{dsfont,textcomp,graphicx}
\usepackage[usenames,dvipsnames]{color}

\usepackage{hyperref}
\hypersetup{
   colorlinks   = true, 
   urlcolor     = blue, 
   linkcolor    = blue, 
   citecolor   = red 
}

\theoremstyle{plain}
\newtheorem{theorem}{Theorem}[section]
    \newtheorem{lemma}[theorem]{Lemma}
    \newtheorem{corollary}[theorem]{Corollary}
    \newtheorem{proposition}[theorem]{Proposition}
    
    \newtheorem{problem}[theorem]{Assertion}
    \newtheorem{example}[theorem]{Example}
    \newtheorem{pr}[theorem]{Problem}
    \newtheorem{oproblem}[theorem]{Open Problem}

\newcommand{\rk}{\operatorname{rk}}

\def\t{\widetilde}
\def\R{{\mathbb R}} \def\Z{{\mathbb Z}}    

\newcommand{\arxivonly}[1]{#1}

\begin{document}
 
\title{Low rank matrix completion and realization of graphs: \\ results and problems} 


\author{S. Dzhenzher, T. Garaev, O. Nikitenko,\\ A. Petukhov, A. Skopenkov, A. Voropaev\footnote{\emph{S. Dzhenzher,  A. Skopenkov:} Moscow Institute of Physics and Technology. 
\emph{T. Garaev:} Moscow State University. 
\emph{O. Nikitenko:} Altay Technical University (Barnaul).
 \emph{A. Petukhov:} Institute for Information Transmission Problems (Moscow). 
\emph{A. Skopenkov:} Independent University of Moscow, \url{https://users.mccme.ru/skopenko/}. 
\emph{A. Voropaev:} (Moscow). 
\newline
We are grateful to E. Kogan for allowing us to use and edit his texts, to Ya. Abramov, E. Alkin, and V. Retinskiy for useful discussions, 
to A. Ryabichev and MCCME publishing house for allowing us to use figures they prepared.}}
  
\date{}

\maketitle

\begin{abstract}
The Netflix problem (from machine learning) asks the following. 
Given a ratings matrix in which each entry $(i,j)$ represents the rating of movie $j$ by customer $i$, if customer $i$ has watched movie $j$, and is otherwise missing, we would like to predict the remaining entries in order to make good recommendations to customers on what to watch next. 
The remaining entries are predicted so as to minimize the {\it rank} of the completed matrix.

In this survey we study a more general problem, in which instead of knowing specific matrix elements, we know linear relations on such elements.  
We describe applications of these results to embeddings of graphs in surfaces (more precisely,  embeddings with rotation systems, and embeddings modulo 2). 
\end{abstract}

\tableofcontents

\section{Motivation and some main results}\label{s:intrnf}


\textbf{Remark} (motivation; formally not used later)
\emph{`Matrix completion is the task of filling in the missing entries of a partially observed matrix...
One example is the movie-ratings matrix, as appears in the Netflix problem (from machine learning): Given a ratings matrix in which each entry $(i,j)$ represents the rating of movie $j$ by customer $i$, if customer $i$ has watched movie $j$ and is otherwise missing, we would like to predict the remaining entries in order to make good recommendations to customers on what to watch next...'}. 
The remaining entries are predicted so as to minimize the {\it rank} of the completed matrix.
All the required definitions (of rank etc.) are given below.
For a brief overview of the history of this and related problems, see \cite{MC, NKS}, \cite[Remark 4]{Ko21}.
 
Here for simplicity we consider matrices with entries in the set $\Z_2=\{0,1\}$ of all residues modulo 2 (with the sum and product operations). 
This is sufficient for the topological applications, see below. 
We start with interesting elementary results in linear algebra. 
They allow us to construct algorithms estimating minimal rank for the particular case of unknown elements {\it on the diagonal} 
(Proposition \ref{t:rk1al} and Theorems \ref{main}, \ref{main-approx}, see also Proposition \ref{p:nondeg}).
Then we study a more general problem, in which instead of knowing specific matrix elements, we know linear relations on such elements.
We estimate the minimal rank of matrices with such relations (Theorems~\ref{p:rank3} and \ref{p:rank}).
 
These results have applications to embeddings of graphs in surfaces (and of $k$-dimensional `hypergraphs' in $2k$-dimensional surfaces). 
See \S\ref{s:weak}, \S\ref{s:emo2}, \cite{DS22, Sk24}.  

\medskip 
Denote by $\Z_2^{s\times n}=(\Z_2^s)^n$ the set of all $s\times n$ matrices with entries in $\Z_2$. 
Let $M_{i,j}$ be the entry in matrix $M$ in the row $i$ and column $j$.
Denote $[n]:=\{1,2,\ldots,n\}$. 


\begin{proposition}[\cite{Bi20}]\label{t:rk1al}  
(a) For a symmetric matrix with $\Z_2$-entries the following conditions are equivalent: 

$\bullet$ some entries on the main diagonal can be changed so that in the resulting matrix all non-zero rows are equal; 

$\bullet$ it is impossible to make the same permutation of rows and of columns\footnote{This means that the rows and columns are numbered by $1,\ldots,n$ (where $n=3,4$) and the permutation of the set $[n]$ is applied both to the rows and to the columns.} so that the upper left square will be one of the submatrices     
$$\begin{pmatrix}
    *& 1& 1\\
    1& *& 0\\
    1& 0& *
\end{pmatrix} \quad\text{or}\quad
\begin{pmatrix}
    *& 1& 0& 0\\
    1& *& 0& 0\\
    0& 0& *& 1\\
    0& 0& 1& *\\
\end{pmatrix},$$
where by * are denoted arbitrary (possibly different) elements.
 
(b) There is an algorithm with the complexity of $O(n^2)$ deciding for a symmetric matrix $M\in\Z_2^{n\times n}$ whether some entries on the main diagonal can be changed so that in the resulting matrix all non-zero rows are equal.
\end{proposition}
 
The algorithmic results in this text could be omitted by theoretically-minded readers because they are easy corollaries of mathematical results.
The {\it complexity} of an algorithm is the number `elementary' steps in this algorithm.
An algorithm has complexity $O(f(n))$ if there is $C>0$ such that the complexity does not exceed $Cf(n)$ for any $n$. 

A square matrix $M\in\Z_2^{n\times n}$ is called \textbf{degenerate} if the sum of its several columns (a non-zero number of columns) is the zero column (i.~e., the column consisting of zeroes only).
A matrix is called \textbf{non-degenerate} otherwise.
\arxivonly{
E.g. of the following matrices $A_1$, $A_2$ and $A_3$ are degenerate, while $A_4$ is non-degenerate.
     \[
    A_1= \begin{pmatrix}0&0&0\\0&0&0\\0&0&0\end{pmatrix},\ 
    A_2=\begin{pmatrix}1&1&1\\1&1&1\\1&1&1\end{pmatrix},\ 
    A_3=\begin{pmatrix}1&1&1\\1&0&1\\1&1&1\end{pmatrix},\ 
    A_4=\begin{pmatrix}0&1&1\\0&0&1\\1&0&0\end{pmatrix}.
    \]
Introduction on degenerate matrices useful for the following result is presented in \S\ref{s:nondeg}. 


For $A_1,A_2,A_3$ we show how to change entries on the main diagonal to make the matrix non-degenerate; we show the opposite for $A_4$:   
\[
    A_1\to \begin{pmatrix}1&0&0\\0&1&0\\0&0&1\end{pmatrix},\quad
    A_2, A_3\to \begin{pmatrix}0&1&1\\1&0&1\\1&1&1\end{pmatrix},\quad
    A_4\to \begin{pmatrix}0&1&1\\0&1&1\\1&0&0\end{pmatrix}.
\]
}

\begin{proposition}\label{p:nondeg} (a) For any matrix $M \in \Z_2^{n\times n}$ some entries on the main diagonal can be changed so that the resulting matrix is degenerate.

(b) The same with `degenerate' replaced by `non-degenerate'.
\end{proposition}

\begin{proof}
(a) Change the numbers on the main diagonal of $M$ so that the sum of the entries in each row is even. 
The resulting matrix is degenerate.

(b) Assume by induction on $n$ that the left upper $(n-1)\times(n-1)$-corner submatrix of $M$ is non-degenerate. 
Apply the decomposition formula for $\det M$ by the last row\arxivonly{ (Assertion \ref{pm:detdec}.c)}. 
We obtain $\det M=M_{n,n}+a$ for some $a\in\Z_2$. 
Thus taking $M_{n,n}=a+1$ we change $M$ so that the resulting matrix is non-degenerate (cf. Lemma \ref{non-degenerate-algo}.b). 
\end{proof}

The \textbf{rank} $\rk M$ of a matrix $M\in\Z_2^{s\times n}$ is the maximal number of columns of $M$ none of whose sums is zero.  
(This is the {\it dimension} of the {\it vector space} formed by the columns of the matrix.)
\arxivonly{
E.g. $\rk A_1=0$, $\rk A_1=1$, $\rk A_3=\rk A_4= 3$, 
and ranks of the following matrices 
    \[
        \begin{pmatrix}0&0&0&0\\0&0&0&0\\0&0&0&0\end{pmatrix},\quad
        \begin{pmatrix}1&1&1&1\\1&1&1&1\\1&1&1&1\end{pmatrix},\quad
        \begin{pmatrix}1&1&1&0\\1&0&1&0\\1&1&1&1\end{pmatrix},\quad
        \begin{pmatrix}0&1&1&1\\0&0&1&1\\1&0&0&1\end{pmatrix}.
    \]
are  0, 1, 3, 3, respectively. 
Introduction on rank useful for the following results is presented in \S\ref{s:lowr}.}

For a matrix $M\in\Z_2^{n\times n}$ let $R(M)$ be the minimal rank of all the matrices obtained by changing some entries on the main diagonal of $M$.
\arxivonly{E.g. for the matrices $A_1,A_2,A_3,A_4$ we have $R(M)$ is 0, 1, 1, 2, respectively.
The number $R(M)$ is not necessarily preserved by

$\bullet$ permutation of columns, because 
$R\!\begin{pmatrix}1&0&1\\0&0&1\\0&1&0\end{pmatrix}=2$ but  $R\!\begin{pmatrix}1&0&1\\0&1&0\\0&0&1\end{pmatrix}=1$.

$\bullet$ adding one column to another one, because
$R\!\begin{pmatrix}1&0&1\\0&1&1\\0&1&0\end{pmatrix}=2$ but $R\!\begin{pmatrix}1&0&1\\0&1&0\\0&0&1\end{pmatrix}=1$.}

\begin{theorem}[\cite{Ko21}]\label{main}
(a) To make a square matrix of rank $k$ out of a square matrix of rank $n$ by changing some diagonal entries, one needs to change at least $|n-k|$ entries.

(b) For any fixed $k$ there is an algorithm with the complexity of $O(n^{k+3})$ deciding for a matrix $M\in\Z_2^{n\times n}$ whether $R(M)\le k$.
\end{theorem}

\smallskip
\begin{proof}[Sketch of a proof (see the details in \S\ref{s:lowr})]
A matrix is said to be \textbf{diagonal} if all its entries outside of the main diagonal are zeroes.
Any matrix obtained by changing some diagonal elements of a matrix $M\in\Z_2^{n\times n}$ can be uniquely represented as the sum $M+D$, where $D$ is a diagonal matrix.
(So the inequality $R(M)\le k$ is equivalent to the existence of a diagonal matrix $D$ such that $\rk(M+D)\le k$.)

Part (a) is easily implied by the subadditivity of rank (Lemma \ref{rk-est}). 

(b) By (a), for $M$ non-degenerate the inequality $R(M)\le k$ is equivalent to the existence of a diagonal matrix $D$ with at most $k$ zeroes on the main diagonal such that $\rk(M+D)\le k$.
The algorithm of (b) constructs a non-degenerate matrix $\overline M$ from $M$ using   
Proposition \ref{p:nondeg}.b, and then adds to $\overline M$ every diagonal matrix with at most $k$ zeroes on the main diagonal.
\end{proof}

The \textbf{identity matrix} $E$ is the diagonal matrix whose diagonal elements are units.

\begin{theorem}[\cite{Ko21}, proved in \S\ref{s:lowr}]\label{main-approx} 
(a) For any non-degenerate matrix $M\in\Z_2^{n\times n}$ and diagonal matrix $D\in\Z_2^{n\times n}$ we have $2\rk(M+D)\ge\rk(M+E)$.

(b) There is an algorithm with the complexity of $O(n^4)$ calculating for a matrix $M\in\Z_2^{n\times n}$ an integer $k$ such that $k/2\le R(M)\le k$.
\end{theorem}

\section{Reminder: degenerate matrices}\label{s:nondeg}
 
\begin{problem}\label{p:det} (a) Degeneracy is not changed under permutation of columns (or rows).
    
(b) Degeneracy is not changed under adding one column (or row) to another.
    
(c) Any matrix can be changed to a diagonal matrix by transformations from (a,b).
    
(d)  A matrix is degenerate if and only if it cannot be changed by transformations from (a,b) to the identity matrix.
    
    
(e) A square matrix is degenerate if and only if the sum of its several rows (a non-zero number of rows) is the zero row.
    
(f) There is an algorithm with the complexity of $O(n^3)$ checking the degeneracy of an $n\times n$ matrix.
\end{problem}
 
For a matrix $M \in \Z_2^{n\times n}$ define the {\it determinant} of $M$ by $\det M:=0$ if $M$ is degenerate, and $\det M:=1$ otherwise. 
Another notation is 
$$\det\!\begin{pmatrix}M_{1,1}&M_{1,2}\\M_{2,1}&M_{2,2}\end{pmatrix}\! = \begin{vmatrix}M_{1,1}&M_{1,2}\\M_{2,1}&M_{2,2}\end{vmatrix}, \quad
\det M=\begin{vmatrix}
    M_{1,1} & \dots & M_{1,n} \\
    \vdots & \ddots & \vdots \\
    M_{n,1} & \dots & M_{n,n}
\end{vmatrix}.$$

\begin{problem}\label{pm:detdec}
(a) $\begin{vmatrix}a&b\\c&d\end{vmatrix} = ad+bc$.

\arxivonly{
(b) $\det(a_1+b_1,a_2,\ldots,a_n)=\det(a_1,a_2,\ldots,a_n)+\det(b_1,a_2,\ldots,a_n)$.
Here and below $a_j,b_1\in\Z_2^n$ are columns of length $n$.

(c) $\det(a_1,\ldots,a_n) = \sum\limits_{i=1}^n a_{i,n} \det(a_1^-, \ldots, a_{i-1}^-, a_{i+1}^-, \ldots, a_n^-)$,
where every column $a_i^- \in \Z_2^{n-1}$ is obtained from the column $a_i$ by deleting the last coordinate.

(d) $\det M = \sum\limits_{\sigma\in S_n} \prod\limits_{i=1}^n M_{i,\sigma(i)}$, where $S_n$ is the set of all permutations (i.~e., 1--1 correspondences) $\sigma:[n]\to[n]$.
}
\end{problem}
 
\begin{lemma}\label{non-degenerate-algo}
(a) Let $M\in \Z_2^{n\times n}$ be a matrix with zeroes on the main diagonal.
Define the sequence $M^{(i)}$, $i = 0, 1, 2, \ldots, n$ recursively as follows: 

$\bullet$ $M^{(0)}:=M$, and 

$\bullet$ $M^{(i)}$ is the result of replacing in $M^{(i-1)}$ the element $M^{(i-1)}_{i,i}=0$ by $1+\delta_i$, where $\delta_0=0$ and $\delta_i:=\det M^{(i-1)}_{[i]\times[i]}$ is the determinant of the left upper $i\times i$-corner submatrix of $M^{(i-1)}$.

Then the matrix $M^{(n)}$ is non-degenerate.

\arxivonly{
(b) There is an algorithm with the complexity of $O(n^4)$ which for a matrix $M\in\Z_2^{n\times n}$ finds some numbers from $\Z_2$ to replace the entries on the main diagonal of $M$ so that the resulting matrix is non-degenerate.
}
\end{lemma}

\emph{Hints and sketches of some proofs.}

\smallskip
\textbf{\ref{p:det}.} Hints: (a)-(b) track the maximal non-degenerate submatrix; (c) use induction. 

Part (a) is clear. 

(b) For a matrix $M$ denote by ${\rm row}_{i\to i+j}M$ the matrix obtained from $M$ by replacing  the $i$th row by the sum of the $i$th row and the $j$th row. 
The matrix ${\rm col}_{i\to i+j}M$ is defined similarly. 

It is clear that to prove part (b) we have to show that the matrices $M$, ${\rm row}_{i\to i+j}M$ and  ${\rm col}_{i\to i+j}M$ are degenerate or not simultaneously. 

Next, observe that ${\rm row}_{i\to i+j}{\rm row}_{i\to i+j}M=M={\rm col}_{i\to i+j}{\rm col}_{i\to i+j}M$. 
Thus it suffices to prove that if $M$ is degenerate then both ${\rm row}_{i\to i+j}M$ and ${\rm col}_{i\to i+j}M$ are degenerate. 

Assume that the sum of columns $c_1,c_2,\ldots,c_s$ of $M$ equals zero.

Then the sum of the `same' columns of ${\rm row}_{i\to i+j}M$ equals zero. 

If $i\notin\{c_1,\ldots, c_s\}$ then the sum of `the same' columns of ${\rm col}_{i\to i+j}M$ equals zero. If $i,j\in\{c_1,\ldots, c_s\}$ then the sum of columns indexed by $\{c_1,\ldots,c_s\}-\{j\}$ of ${\rm col}_{i\to i+j}M$ equals zero. 
If $i\in\{c_1,\ldots,c_s\}$ and $j\notin\{c_1,\ldots,c_s\}$ then the sum of columns indexed by $\{c_1,\ldots,c_s,j\}$ of ${\rm col}_{i\to i+j}M$ equals zero. 

This completes the proof of (b).

(c) We will show explicitly how to produce a diagonal matrix out of $M$. 

If all entries of $M$ are 0 then $M$ is already diagonal. 
If $M$ has a non-zero entry then we place this entry in the top-left corner by  permuting the row of this entry with the top row, and the column of this entry with the left column. 
For the obtained matrix add the top row to other rows and the left column to other columns. 
All entries in the left column and the top row except the top-left entry become zeroes. 
Delete the top row and the left column of the obtained matrix. 

Repeat the procedure inductively for the obtained submatrix.  
In the end this will produce a diagonal matrix. 

(d) By part (c) we can change $M$ into a diagonal matrix using transformations from parts (a, b); also $M$ is degenerate if and only if the new matrix is degenerate. 
It remains to mention that a diagonal matrix is non-degenerate iff it is the identity matrix.

(e) This follows from (a)-(d).

(f) The algorithm is constructed in the solution of part (c).
The algorithm has $n$ major steps, a single major step is described in the second paragraph of that solution. 
Every major step requires at most one permutation of rows, at most one permutation of columns and up to $2n$ additions of rows and columns. 
Thus the complexity of the whole algorithm is $O(n)+n\cdot O(n^2) = O(n^3)$.

\smallskip
\textbf{\ref{pm:detdec}.} (a) The formula follows because a matrix from $\Z_2^{2\times2}$ is degenerate if and only if either it has a zero row, or it has a zero column, or rows are the same and columns are the same (in the latter case all entries are ones).  

Alternatively, here are all matrices from $\Z_2^{2\times2}$ up to permutations of rows and columns:
\[
    \begin{pmatrix}1&0\\0&1\end{pmatrix},
    \begin{pmatrix}1&1\\0&1\end{pmatrix},
    \begin{pmatrix}0&0\\0&0\end{pmatrix},
    \begin{pmatrix}1&1\\1&1\end{pmatrix},
    \begin{pmatrix}1&1\\0&0\end{pmatrix},
    \begin{pmatrix}1&0\\1&0\end{pmatrix},
    \begin{pmatrix}1&0\\0&0\end{pmatrix}.
\]
The first two are non-degenerate, and the others are degenerate.  
It is easy to verify the formula for each of them. 


(d) This follows by (b,c). 
Here is an alternative direct proof. 
Consider $n\times n$ chessboard. 
A {\it rook placement} for such a chessboard is a placement of $n$ rooks on that board with the condition that they do not beat each other.
A {\it rook $M$-placement} for such a chessboard is a rook placement such that all rooks are staying on cells corresponding to unit entries of $M$. 
Denote by $\det^*M$ the parity of the amount of rook $M$-placements.
Then (d) can be restated as follows: $\det M=\det^*M$. 
This follows because 
 
$\bullet$ transformations of \ref{p:det}.a, \ref{p:det}.b preserve $\det^*M$, and 

$\bullet$ $\det M'=\det^* M'$ for a diagonal matrix $M'$. 

\smallskip
\textbf{\ref{non-degenerate-algo}.}
(a) In the following paragraph we prove by induction on $i\ge1$ that the determinant $\Delta_i:=\det M^{(i)}_{[i]\times[i]}$ of the left upper
$i\times i$-corner submatrix of $M^{(i)}$ is equal to $1$.
Then $\det M^{(n)}=\Delta_n=1$.

Base $i=1$ follows because  $\Delta_1 = 1+ \delta_0 = 1$.
Let us prove the inductive step $i-1 \rightarrow i$.
Apply the decomposition formula for the determinant $\Delta_i$ by the last row of the corresponding submatrix of $M^{(i)}$ (Assertion \ref{pm:detdec}.c).
Since $M^{(i-1)}_{i,i}=M_{i,i}=0$ and $\Delta_{i-1}=1$, we have $\Delta_i=\delta_i+(1+\delta_i)\Delta_{i-1}=1$.

(b) The algorithm is constructed by (a).
The algorithm is essentially a computation of the determinants of $n$ square submatrices of sizes $1,2,\ldots,n$.
Hence by Assertion \ref{p:det}.f its complexity is $O(1^3 + 2^3 + \ldots + n^3) = O(n\cdot n^3)= O(n^4)$.

\section{The rank of a matrix}\label{s:lowr}

    \begin{problem}\label{p:rk-base} Take  a matrix $M\in\Z_2^{s\times n}$. 
    
    (a) One can choose $\rk M$ columns of $M$ such that every column is the sum of some chosen columns.
    
(b) Assume that there are $k$ columns (not necessarily of $M$) such that every column of $M$ is the sum of some of them. 
    Then $\rk M\le k$. 

    (c) The rank of a submatrix does not exceed the rank of a matrix.
    \end{problem}

\begin{lemma}[subadditivity of rank]\label{rk-est} 
Let $P,Q$ be matrices of the same size with entries in $\Z_2$.
Then $\rk(P+Q)\le\rk P+\rk Q$. 
\end{lemma}

\begin{proof} Choose columns from Assertion \ref{p:rk-base}.a for $P$ and for $Q$. 
Then every column of $P+Q$ is the sum of some of the chosen $\rk P + \rk Q$ columns. 
By Assertion  \ref{p:rk-base}.b $\rk(P+Q) \le \rk P+\rk Q$.
\end{proof}

\begin{proof}
[Proof of Theorem \ref{main}.a]
Take a matrix $M$ of rank $n$. 
Take a diagonal matrix $Q$ such that $\rk(M+Q)=k$.  
Apply the subadditivity of rank (Lemma \ref{rk-est}) for $P=M+Q$ (then $P+Q=M$). 
We obtain $\rk Q\ge \rk M-\rk (M+Q) = n - k$. 
Analogously $\rk Q\ge k-n$. 
Thus $Q$ has at least $|n-k|$ units on the main diagonal.
\end{proof}

\begin{proof}[Proof of Theorem \ref{main-approx}]
 (a) Denote by $n$ the number of columns of $M$ and of $D$.
    By the subadditivity of rank (Lemma \ref{rk-est}) we have
    \begin{multline*}
        2\rk(M + D) = \rk(M + D) + \rk((M + E) + (E + D)) \\
        \geq (\rk M - \rk D) + (\rk (M + E) - \rk (E + D)) \\
        = n - \rk D + \rk (M + E) - (n - \rk D) = \rk (M + E).
    \end{multline*}

(b) Let $M_n$ be the matrix obtained by applying the algorithm of Lemma \ref{non-degenerate-algo}.b to the matrix $M$.
Let $k := \rk (M_n + E)$.
We have $R(M)=\rk(M+D)$ for some diagonal matrix $D$.
Hence by (a) $k/2 \leq R(M) \leq k$ as required.

The number $k$ can be computed in time $O(n^3)$.
Hence the total complexity of the algorithm is $O(n^4) + O(n^3) = O(n^4)$.
\end{proof}

\arxivonly{

The following is proved similarly to the proof of Assertion \ref{p:det}.


\begin{problem}\label{p:dred} 
(a) A permutation of columns (or of rows) does not change the rank of a matrix.
    
(b) Adding one column to another one (or one row to another one) does not change the rank of a matrix.

(c) The rank of a matrix is equal to the maximal number of its rows none of whose sums is zero.

(d) The rank of a matrix is equal to the maximal size of its non-degenerate square submatrix.
\end{problem}
 
A square matrix with $\Z_2$-entries is called \textbf{even} if all the entries on the main diagonal are zeroes.
    
    \begin{problem}\label{t:rk1}
    (a) For $M\in\Z_2^{s\times n}$ all non-zero rows are equal if and only if $\rk M\le1$.
    
    (b) For a symmetric matrix $M\in\Z_2^{n\times n}$ all non-zero rows are equal if and only if by same permutation of rows and of columns it is possible to obtain a matrix whose upper left square is filled by ones, and all other elements are zeroes.
    
    (c) The rank of any non-zero symmetric even matrix is greater than one.
    \end{problem}
}

    \begin{problem}\label{p:rankintr} (a) There is an algorithm with the complexity of $O(n^3)$ which calculates the rank of a matrix from $\Z_2^{s\times n}$, $s\le n$.
    
    
    (b)* For a fixed integer $k$ there is an algorithm with the complexity of $O(n^2)$ deciding for $M\in \Z_2^{s\times n}$, $s\le n$, whether $\rk M\le k$.
    \end{problem}
 
    \begin{problem}\label{p: kog}  
    There is an algorithm with the complexity of $O(n^{k+3})$ finding for  $M\in\Z_2^{n \times n}$ a diagonal matrix $D$ such that
    
    (a) $\rk(M+D)\le k$; \quad (b) $\rk(M+D)=k$
    
    under the assumption that such a matrix $D$ exists.
    \end{problem}

\begin{proof} [Proof of Theorem \ref{main}.b]
The algorithm is constructed using Theorem \ref{main}.a and Lemma \ref{non-degenerate-algo}.b.
The algorithm given by Lemma \ref{non-degenerate-algo}.b has complexity $O(n^4)$.
There is an algorithm searching through all diagonal $n \times n$ matrices with $\leq k$ zeroes on the main diagonal with the complexity of
$$O\left(n\binom n0 + n\binom n1 + \ldots + n\binom nk\right) \overset{(*)}= O\left((k + 1)n\binom nk\right) =
O\left(n\cdot n^k\right) = O\left(n^{k+1}\right).$$
Here (*) holds because we may assume that $n\ge2k$.
Thus, by Assertion \ref{p:rankintr}.b the complexity of the whole algorithm is $O(n^4)+O(n^{k+1}n^2)=O(n^{k+3})$ (since $k\ge1$).
\end{proof}

\arxivonly{
    \begin{oproblem}\label{p:sharp}
    Is it correct that for any $m,k\le n$ and a matrix $M\in\Z_2^{n\times n}$ of rank $m$, if a matrix of rank $k$ can be obtained by changing some entries on the diagonal of $M$, then this can be done by changing exactly $|m-k|$ entries?
    \end{oproblem}

\begin{pr}\label{p:count}\arxivonly{(a,b,c)} 
Find the number of matrices of rank $k$ in $\Z_2^{n\times n}$ for $k=0,1,2$.
\end{pr}
}

\emph{Hints and sketches of some proofs.}

\smallskip
\textbf{\ref{p:rk-base}.} Hint to (b): find the number of sums of $k$ columns. 

Part (a) follows from the definition of $\rk M$. 

(b) By definition of $\rk$ the number of different sums of columns of $M$ is $2^{\rk M}$. 
On the other hand the number of such sums does not exceed $2^k$. 
Therefore $2^k\ge 2^{\rk M}$, hence $k\ge\rk M$.

\smallskip
\textbf{\ref{t:rk1}.}
Part (a) is clear. 

(b) If for a non-zero symmetric matrix $M$ there exists such a permutation of rows and columns, then $\rk M=1$ by Assertion~\ref{p:rk-base}.b. 

We now take a symmetric matrix $M$ of rank 1.   
As the required permutation we can take any permutation mapping non-zero rows of $M$ to the first rows. 
Indeed, take any non-zero rows $i,j$. 
If $M_{i,j}=0$ then there exists a non-zero row $k$ such that $M_{i,k}=1$. Hence the $j$th and the $k$th rows are distinct non-zero rows of the matrix $M$ of rank one. A contradiction.  
Hence $M_{i,j}=1$. 

(c) Pick a nonzero row and apply the above argument.

\smallskip
\textbf{\ref{p:rankintr}.} Hint for (a): cf. Assertion~\ref{p:det}.

(a) The algorithm from the proof of \ref{p:det}.c provides a diagonal matrix of the same rank, and has the required complexity.  
The rank of a diagonal matrix is equal to the number of non-zero entries in it.

(b) We shall construct a set $S_k$ of columns such that

$\bullet$ these columns constitute a non-degenerate submatrix;

$\bullet$ the first $k$ columns of the matrix $M$ are sums of several columns from the set $S_k$.

If $|S_k|>r$ for some $k=1,\ldots, n$ then the answer is `NO'. 
If for all $k=1,\ldots, n$ we have $|S_k|\le r$ then the answer is `Yes'. 
The answer is correct because $|S_1|\le |S_2|\le \ldots\le |S_n|$, and because $|S_n|>r$ is equivalent to $\rk M>r$.

Set $S_1:=\emptyset$ if the first column of $M$ is 0 then, and $S_1:=\{1\}$ otherwise. 

Let us define $S_{k+1}$ from $S_k$.  
We form the set of all sums of columns of $M$ with indices from $S_k$ (it takes $O(n)$ operations because $|S_k|\le r$). 
Then we compare the $(k+1)$st column of $M$ with all sums from this set (this will take at most $2^r O(n)=O(n)$ operations). 
If the $(k+1)$st column of $M$ is equal to at least one of the sums then $S_{k+1}:=S_k$. 
Otherwise we set $S_{k+1}:=S_k\cup\{k+1\}$. 

It is easy to verify that the total complexity of the algorithm is $O(n^2)$.

\smallskip
\textbf{\ref{p:count}.}
{\it Answers:}  (a) 1; \quad (b) $(2^n-1)^2$; \quad (c) $(2^n-1)^2(2^n-2)^2/6$.

(a) There exists only one matrix of rank 0, the matrix all of whose entries are zeroes.

(b) For the matrix of rank 1 all columns containing a non-zero entry are the same.
Hence such matrices are in 1--1 correspondence with ordered pairs formed by

$\bullet$ a non-empty subset of the set of columns (`non-zero columns'), and

$\bullet$ a nonzero vector in $v\in\Z_2^n$ (`column vector').

Therefore there are $(2^n-1)^2$ such matrices.

(c) Fix a matrix $M$ of rank $2$.
Then there exists a pair $(v,w)$ of columns of $M$ forming a non-degenerate matrix.
Any other column is either $0$ or $v$, or $w$, or $v+w$ (see Assertion \ref{p:rk-base}).
This set $S=S_M$ of four vectors does not depend on a choice of the two columns $v,w$; we call it the {\it column span} of $M$.
(This is a 2-dimensional vector subspace of $\Z_2^n$.)

Each column span is defined by any ordered pair of non-zero vectors in it.
Each column span contains exactly 6 such ordered pairs.
Hence there are $(2^n-1)(2^n-2)/6$ column spans.
Below we prove that 

(A) there are exactly $(2^n-1)(2^n-2)$ matrices of rank 2 for a given column span.

Hence there are $(2^n-1)^2(2^n-2)^2/6$ matrices of rank 2.

{\it First proof of (A).} To a matrix $M$ there correspond the set $X$ of columns of $M$ equal to $v$ or to $v+w$, and the set $Y$ of columns of $M$ equal $w$ or $v+w$. 
Since $\rk M=2$, both sets are non-empty and $X\ne Y$. 
Moreover, matrix $M$ can be reconstructed from $X,Y$. There are $(2^n-1)(2^n-2)$ pairs $(X,Y)$ of distinct non-empty subsets.

{\it Second proof of (A).}
For a 4-element set $S=\{0, v, w, v+w\}$, regard a matrix $M$ with the column span $S$ as a map $\phi_M$ from the set $[n]$ of columns to $S$. 
So $\rk M=2$ if and only if  

(*) the image of $\phi_M$ contains at least two of the vectors $v,w,v+w$.

There are $4^n$ maps $[n]\to S$. 
There are $2^n$ maps $[n]\to\{0,v\}$. 
The same holds for $\{0,v\}$ replaced either by $\{0, w\}$ or by $\{0, v+w\}$. 
There is only one map $[n]\to\{0\}$.
Hence there are exactly $4^n-3\cdot2^n+2=(2^n-1)(2^n-2)$ maps satisfying the condition (*).

\smallskip
{\it Remark.} \cite[Theorem 7.1.5 in p. 299]{HJ20} More generally, the number of matrices of rank~$k$ in~$\Z_2^{m\times n}$ equals 
$$\dfrac{2^{k(k-1)/2}\Pi_{i=0}^{k-1} (2^{m-i}-1)\Pi_{i=0}^{k-1} (2^{n-i}-1)}{\Pi_{i=0}^{k-1} (2^{k-i}-1)}.$$ 

\section{Weak realizability of graphs in surfaces}\label{s:weak}
 
A \textit{hieroglyph} on $n$ letters is a non-oriented cyclic letter sequence of length $2n$ such that each letter from the sequence appears in the sequence twice.

Take a hieroglyph on $n$ letters.
Take a convex polygon with $2n$ sides.
Put the letters in the hieroglyph on the sides of the convex polygon in the non-oriented cyclic order. 
For each letter, glue the ends of a ribbon to the pair of sides corresponding to the letter so that the glued ribbons are pairwise disjoint. 
The ribbons can be either twisted or not twisted.
Call the resulting surface a \textit{disk with ribbons} corresponding to the hieroglyph (Figure \ref{fig:disk-example}).

\begin{figure}[h]\centering
\includegraphics{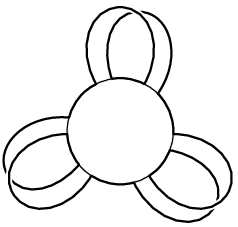}\qquad\qquad\includegraphics{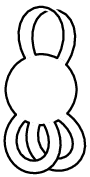}
\caption{Disk with ribbons corresponding to the hieroglyph $aabbcc$ (left) and $aabcbc$ (right)} 
\label{fig:disk-example}
\end{figure}

A hieroglyph $H$ is called \textit{weakly realizable} on the M\"obius band 
if some disk with ribbons corresponding to $H$ can be cut out of the  M\"obius band. 
Analogously, one defines weak realizability on 
{\it the disk with $k$ M\"obius bands} 
(Figure \ref{fig:disk-example}, left)\footnote{Since any connected punctured nonorientable 2-surface of nonorientable genus $m$ is homeomorphic to the disk with $k$ Mobius bands, in this definition the term `disk with $m$ Mobius bands' can be replaced by `a closed connected non-orientable 2-surface of nonorientable genus $k$'.}. 
This is the union of a disk and $k$ pairwise disjoint ribbons having their ends glued to $2k$ pairwise disjoint arcs on the boundary circle of the disk (the ribbons do not have to lie in the plane of the disk) so that

$\bullet$ the orientations of the ends of each ribbon given by an orientation of the boundary circle of the disk have `the same direction along the ribbon', and

$\bullet$ the ribbons are `separated', i.~e. there are $k$ pairwise disjoint arcs $A_i$ on the boundary circle of the disk such that the ends of the $i$-th ribbon are glued to two disjoint arcs contained in $A_i$, $i=1,2,\ldots,k$.


\begin{theorem}\label{t:weak} 
(a) \cite{Bi20} There is an algorithm with the complexity of $O(n^2)$ deciding whether a hieroglyph with $n$ ribbons is weakly realizable on the M\"obius band.

(b) \cite{Ko21} For any fixed $k$ there is an algorithm with the complexity of $O(n^{k+3})$ deciding whether a hieroglyph with $n$ ribbons is weakly realizable on the disk with $k$ M\"obius bands.
\end{theorem}

This is proved using a linear algebraic weak realizability criterion 
(Theorem \ref{t:mohar} below), and a linear algebraic argument (Proposition \ref{t:rk1al} and Theorem \ref{main}). 

Two letters $a,b$ in a hieroglyph $H$ \textit{overlap in $H$} if they interlace in the cyclic sequence of the hieroglyph (i.~e., if they appear in the sequence in the order $abab$ but not $aabb$).
Define the \textit{overlap matrix} $M(H)\in\Z_2^{n\times n}$ of a hieroglyph $H$ as follows. 
Put zeroes on the main diagonal.
Put $1$ in the cell $(i, j)$ for $i \neq j$ if the letters $i, j$ overlap in $H$, and put $0$ otherwise.
 
\begin{theorem}\label{t:mohar}
Hieroglyph $H$ is weakly realizable on the disk with $k$ M\"obius bands if and only if $R(M(H))\le k$. 
\end{theorem}

This follows from \cite[Theorem 3.1]{Mo89} (see also \cite[\S 2.8, statement 2.8.8(c)]{Sk20}).

\arxivonly{See more in \cite{Bi20}, \cite[Appendix]{Ko21}, \cite[\S2]{Sk20}.}

\section{Modulo 2 embeddings of graphs to surfaces}\label{s:emo2}

This section is formally not used later, but serves as an additional motivation for \S\ref{s:comhyp}.  

Denote by $S$ the torus, or sphere with handles, or the M\"obius band, or the Klein bottle, or even any $2$-dimensional surface. 
Their simple definitions can be found e.g. in \cite[\S2.1]{Sk20}. 


Below graph drawings on $S$ may have self-intersections. 
An {\it embedding} (or realization) is a graph drawing without self-intersections. 

\begin{figure}[h]\centering
\includegraphics[scale=.9]{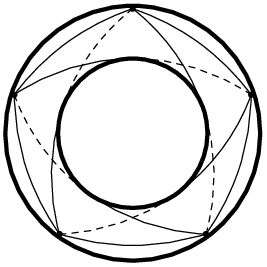}\qquad\qquad
\includegraphics[scale=1.2]{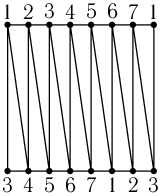}\qquad\qquad
\includegraphics{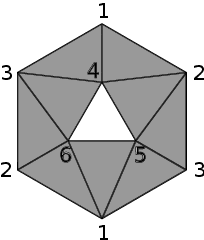}
\caption{Realization of nonplanar graphs}\label{real}
\end{figure}

\begin{example}\label{p:draw}    
(a) Beautiful realizations of the graphs~$K_5$ and $K_7$ on the torus are shown in Figure~\ref{real}, left and middle. 

(b) A beautiful realization of $K_6$ in the M\"obius band is shown in Figure.~\ref{real}, right.  

(c) There is an embedding of $K_m$ in the sphere with some number of handles (depending on $m$). 


Draw the graph~$K_m$ in the plane with 
only double self-intersection points. 
In a~small neighborhood of every double point, attach a handle and lift one of the edges `bridgelike' over the other edge to the handle (Figure~\ref{f:res}). 
\end{example}

\begin{figure}[h]\centering
\includegraphics{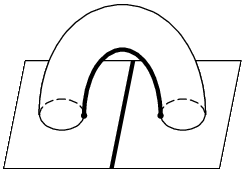} 
\caption{Resolving intersection by adding a handle}
\label{f:res}
\end{figure}

\begin{figure}[h]\centering
\includegraphics[scale=.9]{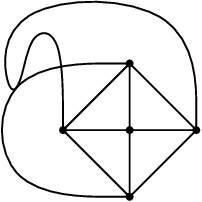}
\caption{
A `non-general position even drawing' of $K_5$ in the plane. 
The drawings (i.~e., the images of) every two non-adjacent edges intersect at an even number of points.}
\label{f:evenk5}
\end{figure}
 
A {\it self-intersection point} of a drawing is a point on the drawing to which corresponds more than one point of the graph itself. 

A graph drawing is said to be \textbf{general position} if 

$\bullet$ to every self-intersection point there correspond exactly two points of the graph; 

$\bullet$ every drawing of a vertex is not a self-intersection point, 

$\bullet$ the drawing has finitely many  self-intersection points, and 

\begin{figure}[h]\centering
\includegraphics{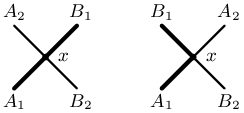}
\caption{A transverse intersection and a non-transverse intersection}
\label{trans}
\end{figure}

$\bullet$ at every such point the self-intersection is transverse (Figure \ref{trans})\footnote{Strictly speaking, the transversality is only easy to define for PL (piecewise-linear) not continuous graph drawings. 
PL curves on the torus can be easily defined by regarding the torus as obtained from a rectangle by gluing.
A \emph{PL curve on the torus} is then a family of polygonal lines in the rectangle satisfying certain conditions (work out these conditions!).
In a similar way, other surfaces $S$ can be obtained from plane polygons by gluing. 
This allows one to define PL curves on $S$. A graph drawing on $S$ is called {\it PL} if the drawing of every edge is PL.   
Another formalizations are given in \cite[\S4, \S5]{Sk20}.}.  

A general position graph drawing is a \textbf{$\Z_2$-embedding} if the drawings of every two non-adjacent edges intersect at an even number of points. 

If $S$ is either the plane or the torus or the M\"obius band, and a graph has a $\Z_2$-embedding to $S$, then the graph embeds to $S$ (Hanani-Tutte; Fulek-Pelsmajer-Schaefer, 2020; Pelsmajer-Schaefer-Stasi, 2009). 
However, there is a graph having a $\Z_2$-embedding to the sphere with 4 handles but not an embedding in the sphere with 4 handles (Fulek-Kyn\v cl, 2017). 
See references in \cite[Remark 1.3.b,c]{Bi21}.

\begin{theorem}[\cite{FK19}]\label{t:nonz2}
If a graph $K$ has a $\Z_2$-embedding to the sphere with $g$ handles, then 


(a) $g\ge(m-5)^2/16$ for $K=K_m$. \qquad
(b) $g\ge(n-2)^2/4$ for $K=K_{n,n}$.
\end{theorem}

Theorem \ref{t:nonz2} is proved by showing that on a surface to which a large graph has a  $\Z_2$-embedding, the intersections of closed curves are sufficiently complicated (in the sense of rank of a certain matrix; cf. Assertion \ref{r:low}).
More precisely, 

$\bullet$ the weaker estimation $g\ge(m-4)/3$ for $K=K_m$ \cite{PT19} follows by Theorems \ref{t:hbetti}.a and \ref{p:rank3} together with Assertion \ref{r:low} (all below);  

$\bullet$ Theorem \ref{t:nonz2}.a follows by Theorem \ref{t:nonz2}.b (prove!);   

$\bullet$ Theorem \ref{t:nonz2}.b is proved in \cite{FK19} (see a well-structured exposition in \cite{DS22}).  

Analogously, Assertion \ref{r:low} and Theorem \ref{p:rank3}
(together with Theorem \ref{t:hbetti}.b) 
imply the non-$\Z_2$-embeddability of 
$K_7$ to the M\"obius band. (They also imply the non-embeddability of $K_7$ to the Klein bottle, which does not follow from the Euler inequality.) 
There is an analogous non-embeddability result in higher dimensions \cite{DS22}.  

Denote by $|X|_2\in\Z_2$ the parity of the number of elements in a finite set $X$.

Closed curves $\gamma_1,\ldots,\gamma_p$ on $S$ are said to be in \textbf{general position} if the graph drawing (of the disjoint union of $p$ cycles) formed by these curves is in general position. 
Their {\it intersection $p\times p$- matrix} $G$ is defined as 
$$G_{i,j}:=\begin{cases}|\gamma_i\cap\gamma_j|_2, & i\ne j,\\ |\gamma_i\cap\gamma_i'|_2, & i=j,\end{cases}$$ 
where $\gamma_i'$ is a curve close to $\gamma_i$ in general position to $\gamma_i$.

\begin{theorem}[Homology Betti Theorem]\label{t:hbetti} 
For any closed general position curves  on

(a) the sphere with $g$ handles the rank of their intersection matrix does not exceed $2g$. 

(b) the disk with $m$ M\"obius bands the rank of their intersection matrix does not exceed $m$.
\end{theorem}


\begin{figure}[h]\centering
\includegraphics[scale=.8]{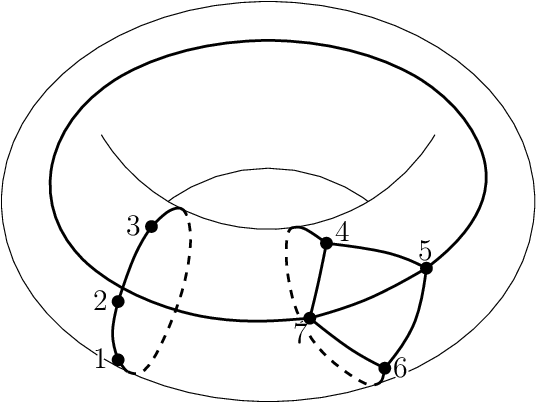}\qquad
\includegraphics[scale=.8]{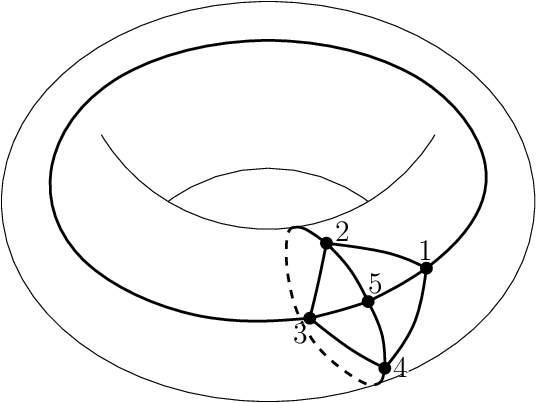}
\caption{Left: $K_3$ and $K_4$ on the torus.
Right: $K_5$ on the torus}
\label{f:relations}
\end{figure}

\begin{problem}\label{r:low} 
Take any embedding (or $\Z_2$-embedding) $f\colon K_n\to S$.  
Take any map $f'\colon K_n\to S$ `in general position' 
to $f$, and close to $f$.
For any pairwise different numbers $i,j,k\in[n]$ denote by $\langle ijk\rangle$ the cycle of length $3$ in $K_n$ passing through $i,j,k$. 
Let 
\[
ijk\wedge pqr := \bigl\lvert f\langle ijk\rangle\cap f'\langle pqr\rangle\bigr\rvert_2.
\]
Then 

(\ref{r:low}.1) \quad $123\wedge456=0$;   

(\ref{r:low}.2) \quad $123\wedge456 + 123\wedge567 + 123\wedge467 + 123\wedge457 = 0$; 
    
\phantom{(\ref{r:low}.2) }\quad $123\wedge345 + 123\wedge346 + 123\wedge356 + 123\wedge456 = 0$.
    
\phantom{(\ref{r:low}.2) }\quad $123\wedge234 + 123\wedge235 + 123\wedge245 +  123\wedge345 = 0$.
    
\phantom{(\ref{r:low}.2) }\quad $123\wedge123 + 123\wedge124 + 123\wedge134 + 123\wedge234 = 0$.

(See Figure \ref{f:relations}, left. 
This follows from $abc\oplus abd\oplus acd\oplus bcd=\emptyset$. 
For one formula covering these four formulas see the linear dependence property of \S\ref{s:comhyp}.)     

(\ref{r:low}.3) \quad $125\wedge345 + 135\wedge245 + 145\wedge235 = 1$.
 
See Figure \ref{f:relations}, right. 
This is easily deduced from (B) below. 
\end{problem}

We denote by $\oplus$ the mod~$2$ summation (i.e., the symmetric difference) of sets. 


\smallskip
\textbf{Remark.}
(A) For any pairwise distinct points $A_1,A_2,A_3,A_4$ in the line there is exactly one `intertwined' coloring into two colors.

(B) For any pairwise distinct points $A_1,A_2,A_3,A_4$ on the circle 
$$\lvert A_1A_2\cap A_3A_4\rvert+\lvert A_1A_3\cap A_2A_4\rvert+\lvert A_1A_4\cap A_2A_3\rvert=1.$$

(B') For any `general position' map $f\colon K_5\to\R^2$ the number of intersection points in $\R^2$ formed by images of disjoint edges is odd. 

A simple deduction of $(A)\Rightarrow(B')$ is presented in \cite{Sk14} (for the linear case; for the PL case the deduction is analogous). 
Observe that (B') does not follow from Euler's formula for planar graphs. 
Analogously, the non-$\Z_2$-embeddability to surfaces (unlike the non-embeddability) does not follow from the Euler inequality for surfaces \cite[\S2.4]{Sk20}. 

\begin{problem}\label{p:inters}
Take any embedding (or $\Z_2$-embedding) $f\colon K_{n,n}\to S$.  
Take any map $f'\colon K_{n,n}\to S$ `in general position' to $f$, and close to $f$.  
Denote by $a'$ a copy of an object $a$, e.~g. the parts of $K_{n,n}$ are $[n]$ and $[n]'$.
For any different numbers $a,i\in[n]$ and $b,j\in[n]$ denote by $\langle aj'ib'\rangle$ the cycle of length $4$ in $K_{n,n}$ passing through $a,j',i,b'$.
Let 
\[
aj'ib'\wedge cl'kd' := \bigl\lvert f\langle aj'ib'\rangle\cap f'\langle cl'kd'\rangle\bigr\rvert_2.
\]
Then 
  
(\ref{p:inters}.1) \quad $32'23'\wedge31'13' + 31'23'\wedge32'13' = 1$.

(\ref{p:inters}.2) For any pairwise distinct $x,y,z\in[n-1]$, the residue $nj'xn'\wedge nl'yn' + nj'zn'\wedge nl'yn'$ is independent of distinct $j,l\in[n-1]$.
\end{problem}


\section{Rank of matrix with relations}\label{s:comhyp}

We shorten $\{i\}$ to $i$.
An {\it ${[m]\choose 3}$-matrix} is a symmetric square matrix with $\Z_2$-entries whose rows and whose columns correspond to all $3$-element subsets of $[m]$, and for which the following properties hold:  

(triviality) $A_{P,Q}=0$ if $P\cap Q=\emptyset$;

(linear dependence) for each $4$-element and $3$-element subsets $F,P\subset [m]$
$$\sum\limits_{i\in F} A_{F-i,P}=0.$$

(non-triviality) for each $i\in[m]$ and $4$-element subset $F\subset [m]-i$ we have $A_{F,i}=1$, where
$$A_{F,i}: = \sum\limits_{\{X,Y\}\ :\ F\cup i=X\cup Y,\ X\cap Y=i,\ |X|=|Y|=3} A_{X,Y}.$$ 
By Assertion \ref{r:low}, an ${[m]\choose 3}$-matrix is constructed from a $\Z_2$-embedding $f\colon K_m\to S$ to a surface. 
Indeed, set  $A^f_{\{i,j,k\},\{p,q,r\}}:=ijk\wedge pqr$. 
If the surface $S$ is orientable, then the constructed matrix $A^f$ is even (i.~e., $A^f_{P,P}=0$ for each $3$-element subset $P\subset [m]$).



\begin{theorem}[\cite{PT19}]\label{p:rank3}
(a) If $A$ is an ${[m]\choose 3}$-matrix, then $\rk A\ge\dfrac{m-4}3$.

(b) If, moreover, $A$ is even, then $\rk A\ge\dfrac{2(m-4)}5$.
\end{theorem}

Proofs of the following particular cases of Theorem \ref{p:rank3} lead to the general case.
You can deduce Theorem~\ref{p:rank3} from Proposition~\ref{p:rankest3a}.a,b.
For stronger estimations see \S\ref{s:comhypge}.

\begin{problem}\label{p:motno}
    (a)* There are no ${[7]\choose 3}$-matrices of rank $1$. 
    
    (b)* There are no even ${[8]\choose 3}$-matrices of rank smaller than $3$.
\end{problem}

\begin{example}\label{p:mot}
    (a) There is a non-zero ${[4]\choose 3}$-matrix.
    
    (b) There is a ${[5]\choose 3}$-matrix.
    
    (c) For any $m\ge5$ there is an ${[m]\choose 3}$-matrix.
    
    (a',b',c') The same for even matrices.
    
    (d) There is a ${[5]\choose 3}$-matrix of rank $1$.
    
    (e) There is an even ${[5]\choose 3}$-matrix of rank $2$.
     
    (f) There is a ${[6]\choose 3}$-matrix of rank $1$.
    
    (g)* There is an ${[8]\choose 3}$-matrix of rank greater than $2$.

In the proof one does not need to explicitly give the matrix, just describe the construction.
We only know a proof using Assertions \ref{p:draw}, \ref{r:low}, and Theorem \ref{t:hbetti}.  

(a) Consider the $K_4$-subgraph of $K_5$ with vertices 1, 2, 3, 4 together with its embedding to the torus defined by Figure~\ref{f:relations}, right side. Then $A^f$ is a non-trivial $[m]\choose 3$-matrix. 

(b) Let $f$ be the embedding of $K_5$ into the torus defined by the right side of Figure~\ref{f:relations}. 
Then $A^f$ has the required property. 
 
(c) By Assertion~\ref{p:draw}.c there exists an embedding $f\colon K_m \to S$ where $S$ is a sphere with several handles. 
Then $A^f$ has  the required property. 

(a', b', c') The matrices $A^f$ from (a, b, c) respectively satisfy the needed conditions. 

Parts (d, e, f) follow from Assertion~\ref{pm:tr5}.a, b. 
\end{example}


\begin{problem}[obvious]\label{p:hered3}
    Let $A'$ be the square matrix of size ${m-1\choose 3}$ obtained from an ${[m]\choose 3}$-matrix by deleting rows and columns corresponding to all subsets containing $m$.
    Then $A'$ is an ${[m-1]\choose 3}$-matrix.
\end{problem}

\begin{problem}\label{p:step-alg3}
    (a) Let $B$ be the square matrix of size ${m-3\choose 3}$ obtained from an ${[m]\choose 3}$-matrix $A$ by deleting rows and columns corresponding to subsets containing at least one element of $X:=\{m,m-1,m-2\}$.
    If $A_{X,X}=1$, then $\rk A>\rk B$.

    (b) Let $C$ be the square matrix obtained from an ${[m]\choose 3}$-matrix $A$ by deleting rows and columns corresponding to subsets containing at least one element of certain $3$-element subsets $X,Y\subset[m]$.
    If $A_{X,X}=A_{Y,Y}=0$ and $A_{X,Y}=1$, then $\rk A\ge \rk C+2$.
\end{problem}

Denote by $r_m$ the minimal rank of an ${[m]\choose 3}$-matrix.
Denote by $\t{r_m}$ the minimal rank of an even ${[m]\choose 3}$-matrix.
Clearly, $r_m=\t{r_m}=0$ for $m\le4$, and $r_m\le\t{r_m}$. 
The non-triviality implies that   $r_5,\t{r_5}\ge1$. 
Theorem \ref{p:rank3} asserts that $r_m\ge\dfrac{m-4}3$ and $\t{r_m}\ge\dfrac{2(m-4)}5$.

\begin{pr}\label{pm:tr5} 
    (a,b) Find $r_5,r_6$ and $\t{r_5},\t{r_6},\t{r_7}$.
    
    (c) Both sequences $r_m,\t{r_m}$ are non-decreasing. 
\end{pr}

\begin{proposition}\label{p:rankest3a}
    (a) $r_m\ge\min\{r_{m-3}+1,\t{r_m}\}$ (more precisely, either $r_m=\t{r_m}$ or $r_m\ge r_{m-3}+1$);
    
    (b) $\t{r_m}\ge \t{r_{m-5}}+2$.
    
\end{proposition}



    

\emph{Hints and sketches of some proofs.}

\smallskip
\textbf{\ref{p:rank3}.} (b) Induction on $m$. 
The base $m\le4$ is clear. 
By Assertion~\ref{p:rankest3a}.b and inductive hypothesis we have
\[
    \t{r_m}\ge\t{r_{m-5}}+2\ge \dfrac{2(m-5-4)}5+2 = \dfrac{2(m-4)}5.
\]

(a) Induction on $m$. 
The base $m\le4$ is clear.
By Assertion~\ref{p:rankest3a}.a, Theorem~\ref{p:rank3}.b and inductive hypothesis we have
\[
    r_m\ge\min\left\{r_{m-3}+1,\t{r_m}\right\}\ge \min\left\{\dfrac{m-3-4}3+1,\dfrac{2(m-4)}5\right\} = \dfrac{m-4}3.
\]

\smallskip
\textbf{\ref{pm:tr5}.} (a) Recall that if $m\ge5$ then every $[m]\choose 3$-matrix is not a zero matrix. 
This implies that $r_5, r_6 \geq 1$. 
Let $f$ be an embedding defined by Assertion~\ref{p:draw}.b1, b2. 
It follows from Theorem~\ref{t:hbetti} that $\rk A^f\le 1$. 
Hence $r_5=r_6=1$. 

(b) Assertion~\ref{t:rk1}.c implies that $\t{r_5}, \t{r_6}, \t{r_7}\geq 2$. 
We have $\rk A^f=2$ for $f$ defined by Assertion~\ref{p:draw}.a1, a2, a3 and hence $\t{r_5}=\t{r_6}=\t{r_7}=2$.

(c) Let $A$ be an $[m]\choose 3$-matrix and let $B$ be an $[m-1]\choose 3$-submatrix of $A$ introduced in Assertion~\ref{p:hered3}. We have $\rk B\le \rk A$ and therefore $r_m\geq r_{m-1}$. 

If $A$ is even then $B$ is even and therefore $\t{r_m}\geq \t{r_{m-1}}$.

\smallskip
\textbf{\ref{p:rankest3a}.} (Take $l=3$.)
(a) Take an $[m]\choose l$-matrix $A$ such that $\rk A=r_m$.
If $A$ is even, then $r_m=\t{r_m}$, so we are done.
Otherwise, there is an $l$-element subset $X\subset [m]$ such that $A_{X,X}=1$.
Let $B$ be the `restriction' of $A$ to $l$-element subsets of $[m]-X$.
Then
\[
    r_m=\rk A\ge\rk B+1\ge r_{m-l}+1,    \quad\text{where}
\]

$\bullet$ the first inequality follows by Assertion~\ref{p:step-alg3}.a; 

$\bullet$ the second inequality holds because $B$ is a $[m]-X\choose l$-matrix by Assertion~\ref{p:hered3}.

(b) Take an even $[m]\choose l$-matrix $A$ such that $\rk A=\t{r_m}$.
By the non-triviality $A\ne0$.
Hence there are $l$-element subsets $X,Y\subset[m]$ such that $A_{X,Y}=1$.
Let $C$ be the `restriction' of $A$ to $l$-element subsets of $[m]-X-Y$.
Then
\[
    \t{r_m}=\rk A\ge\rk C+2\ge \t{r_{m-2l+1}}+2,   \quad\text{where}
\]

$\bullet$ the first inequality follows by Assertion~\ref{p:step-alg3}.b;

$\bullet$ the second inequality holds because $C$ is a $[m]-X-Y\choose l$-matrix by Assertion~\ref{p:hered3}, and because $A_{X,Y}=1$, so by the triviality $X\cap Y\ne\emptyset$, hence $|[m]-X-Y|\ge m-2l+1$. 



\section{Rank of matrix with relations: generalization}\label{s:comhypge}

The following results are `higher-dimensional' (and stronger) generalizations of Theorem~\ref{p:rank3}, Assertions~\ref{p:hered3} and~\ref{p:step-alg3}, and Proposition~\ref{p:rankest3a}.
They give a simplified well-structured exposition of \cite[Theorem 1]{PT19}. 

An {\it ${[m]\choose l}$-matrix} is a symmetric square matrix with $\Z_2$-entries whose rows and whose columns correspond to all $l$-element subsets of $[m]$, and for which (triviality)  and the following properties hold:  

(linear dependence) for each $(l+1)$-element and $l$-element subsets $F,P\subset [m]$
$$\sum\limits_{i\in F} A_{F-i,P}=0.$$

(non-triviality) for each $i\in[m]$ and $(2l-2)$-element subset $F\subset [m]-i$ we have $A_{F,i}=1$, where
$$A_{F,i}:= \sum\limits_{\{X,Y\}\ :\ F\cup i=X\cup Y,\ X\cap Y=i,\ |X|=|Y|=l} A_{X,Y} = \sum\limits_{\{\sigma,\tau\}\ :\ F=\sigma\sqcup \tau,\ |\sigma|=l-1} A_{i\sqcup \sigma,i\sqcup\tau}.$$

Analogously to Assertion \ref{r:low}, an ${[m]\choose l}$-matrix is constructed by a $\Z_2$-embedding of the $(l-1)$-dimensional skeleton of the $(m-1)$-dimensional simplex to a $2(l-1)$-dimensional manifold.  

\begin{theorem}[\cite{PT19}]\label{p:rank}
    Suppose $l\ge3$ and $A$ is an ${[m]\choose l}$-matrix. 
    
    (a) Then $\rk A\ge\dfrac{m-2l+2}{l-1}$. \quad
    (b) If, moreover, $A$ is even, then $\rk A\ge\dfrac{2(m-2l+2)}l$.
\end{theorem}

You can deduce Theorem \ref{p:rank} from Propositions \ref{p:rankest}.a,b.

The following Assertions~\ref{p:hered-l} and \ref{p:step-alg-l}.a,b generalize  
Assertions~\ref{p:hered3} and~\ref{p:step-alg3}.

\begin{problem}[obvious]\label{p:hered-l}
    Let $A'$ be the square matrix of size ${m-1\choose l}$ obtained from an ${[m]\choose l}$-matrix by deleting rows and columns corresponding to all subsets containing $m$.
    Then $A'$ is an ${[m-1]\choose l}$-matrix.
\end{problem}
 
\begin{problem}\label{p:step-alg-l}
    Let $A$ be an ${[m]\choose l}$-matrix and $X:=\{m-l+1,m-l+2,\ldots,m\}$.

    (a,b') Let $B$ be the square matrix of size ${m-l\choose l}$ obtained from $A$ by deleting rows and columns corresponding to subsets containing at least one of the elements of $X$.
    
    If $A_{X,X}=1$, then $\rk A>\rk B$.
    
    If $A_{X,X}=A_{Y,Y}=0$ and $A_{X,Y}=1$ for some $Y\subset[m]$, then $\rk A\ge\rk B+2$.
    
    (b) Let $C$ be the square matrix obtained from $A$ by deleting rows and columns corresponding to subsets containing at least one element of $X$ or of certain $l$-element subset $Y\subset[m]$.
    If $A_{X,X}=A_{Y,Y}=0$ and $A_{X,Y}=1$, then $\rk A\ge \rk C+2$.
    
    (a') For $l$-element subsets $P,Q\subset[m-l+1]$ define
    \[
        D_{P,Q}:=A_{P,Q}+A_{P,X}A_{Q,X}.
    \]
    If $A_{X,X} = 1$, then $\rk D<\rk A$ and $D$ is an $[m-l+1]\choose l$-matrix.
\end{problem}

Assertions~\ref{p:step-alg-l}.a,b are only required to illustrate the idea of Assertions~\ref{p:step-alg-l}.a',b' by proving much easier results giving estimates $\rk A\ge\dfrac{m-2l+2}l$ and, for $A$ even, $\rk A\ge\dfrac{2(m-2l+2)}{2l-1}$.

Denote by $r_m$ the minimal rank of an ${[m]\choose l}$-matrix.
Denote by $\t{r_m}$ the minimal rank of an even ${[m]\choose l}$-matrix.
Clearly, $r_m=\t{r_m}=0$ for $m\le2l-2$, both sequences $r_m,\t{r_m}$ are non-decreasing, and $r_m\le\t{r_m}$.
The non-triviality implies that $r_{2l-1},\t{r_{2l-1}}\ge1$.
Theorem~\ref{p:rank} asserts that $r_m\ge\dfrac{m-2l+2}{l-1}$ and $\t{r_m}\ge\dfrac{2(m-2l+2)}l$.

\begin{proposition}[cf. Proposition \ref{p:rankest3a}]\label{p:rankest}
    
    
    (a) $r_m\ge\min\{r_{m-l+1}+1,\t{r_m}\}$ (more precisely, either $r_m=\t{r_m}$, or $r_m\ge r_{m-l+1}+1$);
    
    (b)  $\t{r_m}\ge\t{r_{m-l}}+2$.
\end{proposition}


Proof of Proposition~\ref{p:rankest}.a also uses an algebraic version (b) of the higher-dimensional analogue of the following result (a).

\begin{proposition}\label{p:rankher}
    (a) Denote by $X={{[5]\choose2}\choose2}$ the set of unordered pairs of $2$-element subsets of $[5]$. 
    For any $i\in[5]$ and a partition $[5]-i=\sigma\sqcup\tau$ into disjoint $2$-element sets denote 
    \[
        T_{i,\{\sigma,\tau\}}:=\bigl\{\{\alpha,\beta\}\in X\ :\  \alpha\subset\sigma\sqcup i,\ \beta\subset\tau\sqcup i\bigr\}.
    \]
    Denote by $A_i$ the sum modulo $2$ (i.~e., the symmetric difference) of sets $T_{i,\{\sigma,\tau\}}$ over all non-ordered partitions $[5]-i=\sigma\sqcup\tau$ as above.
    Then
    \[
        A_i=\bigl\{\{\alpha,\beta\}\in X\  :\ \alpha\cap\beta=\emptyset\bigr\}
    \]
    and so is independent of $i$.
    
    (b)
    Let $A$ be a symmetric square matrix with $\Z_2$-entries whose rows and whose columns correspond to all $l$-element subsets of $[m]$.
    If $A$ satisfies the linear dependence property (from the definition of an ${[m]\choose l}$-matrix), then $A_{F,i}$ depends only on $F\sqcup i$ not on $(F,i)$. 
\end{proposition}

\emph{Hints and sketches of some proofs.}
 
\smallskip
\textbf{\ref{p:step-alg-l}.} (For \ref{p:step-alg3} take $l=3$.)

(a) Let $B'$ be the `restriction' of $A$ to $X$ and to $l$-element subsets of $[m]-X$.
Then
\[
    \rk A\ge\rk B'=\rk B+1,
\]
where equality holds because by the triviality $B'_{X,Z}=0$ for any $Z\subset [m]-X$.

(b) Let $C'$ be the `restriction' of $A$ to $X$, $Y$ and $l$-element subsets of $[m]-X-Y$.
Then
\[
    \rk A\ge\rk C'=\rk C+2,
\]
where equality holds because by the triviality $C'_{X,Z}=C'_{Y,Z}=0$ for any $Z\subset [m]-X-Y$.

(b')
Take a basis of $\Z_2^{{m\choose l}}$ corresponding to $l$-element subsets of $[m]$.
Define a bilinear form $A$ on $\Z_2^{{m\choose l}}$ by setting $A(P,Q):=A_{P,Q}$ for basic vectors $P$, $Q$.
Take any $l$-element set $P\subset[m]$.
Let
\[
    \overline P=\overline P(X,Y):=P+A_{X,P}Y+A_{Y,P}X.
\]
Recall that
\begin{equation}
    A_{X,Y}=A_{Y,X}=1\quad\text{and}\quad A_{X,X}=A_{Y,Y}=0.
\tag{*}
\end{equation}
Hence
\begin{equation}
    A(\overline P,X)=A(\overline P,Y)=0
\tag{**}
\end{equation}
(i.~e., $\overline P$ is the orthogonal projection of $P$ to the orthogonal complement of $\left<X,Y\right>$
with respect to $A$).
By the triviality, for $P\subset[m]-X$ we have $\overline P=P+A_{Y,P}X$.
Hence for every $l$-element sets $P,Q\subset [m]-X$ we have
\begin{equation}
    A(\overline P,\overline Q) = A_{P, Q}+0+0+0 = B_{P, Q}.
\tag{***}
\end{equation}
(I.~e., $B$ is the Gramian matrix with respect to $A$ of the `projections' $\overline P$ of $l$-element sets $P\subset[m]-X$.)
Let $B'$ be the Gramian matrix with respect to $A$ of $X$, $Y$ and the `projections' $\overline R$ of $l$-element sets $R\subset[m]-X$.
I.~e., $B'_{P,Q}=A(\widehat P,\widehat Q)$, where $\widehat P=P$ if $P\in\{X,Y\}$, and $\widehat P=\overline P$ otherwise ($\widehat Q$ is defined analogously).
Then

$\bullet$ $B'_{X,Y} = B'_{Y,X} = 1$, $B'_{X, X} = B'_{Y, Y} = 0$ (by (*)),

$\bullet$ $B'_{X,P} = B'_{P,X} = B'_{Y,P} = B'_{P,Y} = 0$ for $P\ne X,Y$ (by (**)), and

$\bullet$ $B'_{P,Q} = B_{P,Q}$ for $P,Q\subset[m]-X$ (by (***)).

Hence $\rk B+2 = \rk B' \le \rk A$.

(a') In this paragraph we prove that $\rk D<\rk A$.
Take a basis of $\Z_2^{{m\choose l}}$ corresponding to $l$-element subsets of $[m]$.
Define a bilinear form $A$ on $\Z_2^{{m\choose l}}$ by setting $A(P,Q):=A_{P,Q}$ for basic vectors $P,Q$.
Let $P_X$ be the orthogonal projection of $P$ to the orthogonal complement of $X$ (with respect to $A$), i.~e., $P_X:=P+A_{P,X}X$.
We have
\begin{multline*}
    A(P_X,Q_X) = A(P,Q)+A(A_{P,X}X,Q)+A(P,A_{Q,X}X)+A(A_{P,X}X,A_{Q,X}X) = \\ =
    A_{P,Q}+A_{P,X}A_{X,Q}+A_{P,X}A_{Q,X}+A_{P,X}A_{Q,X}A_{X,X} = A_{P,Q} + A_{P,X}A_{Q,X} = D_{P,Q}.
\end{multline*}
Then $D$ is the Gramian matrix (with respect to $A$) of the projections of subsets of $[m - l + 1]$.
Let $D'$ be the Gramian matrix (with respect to $A$) of $X$ and the projections of subsets of $[m - l + 1]$.
We have $D_{P, Q} = D'_{P, Q}$ for all subsets $P, Q \subset [m - l + 1]$.
Furthermore, $D'_{X,P} = D'_{P, X} = 0$ for any basic vector $P \ne X$ and $D'_{X, X} = A_{X, X} = 1$.
Thus $\rk D = \rk D' - 1 <\rk A$.

In this paragraph we prove that $D$ satisfies the triviality property.
If $P \cap Q = \emptyset$, then either $P \cap X = \emptyset$, or $Q \cap X = \emptyset$.
Hence $D_{P, Q} = A_{P, Q} + A_{P, X}A_{Q, X} = 0 + 0 = 0$.

In this paragraph we prove that $D$ satisfies the linear dependence property.
For each $(l+1)$-element and $l$-element subsets $F,P\subset [m-l+1]$ we have
\[
    \sum\limits_{i \in F} D_{F - i, P} =
    \sum\limits_{i \in F} A_{F - i, P} + A_{P, X} \sum\limits_{i \in F} A_{F - i, X} = 0.
\]

In this paragraph we prove that $D$ satisfies the non-triviality property.
By Proposition \ref{p:rankher}.b for $D$, we may assume that $i \neq m - l + 1$.
Then for each summand $D_{i\sqcup \sigma,i\sqcup\tau}$ of $D_{F, i}$
at least one of the sets $i\sqcup \sigma,i\sqcup\tau$ does not contain $m - l + 1$ and hence does not intersect $X$.
Hence
$D_{i\sqcup \sigma,i\sqcup\tau} = A_{i\sqcup \sigma,i\sqcup\tau} + A_{i\sqcup \sigma, X}A_{i\sqcup\tau, X} = A_{i\sqcup \sigma,i\sqcup\tau}$.
Thus $D_{F,i}=A_{F,i}=1$.

\smallskip
\textbf{\ref{p:rankest}.}
(a) Take an ${[m]\choose l}$-matrix $A$ such that $\rk A = r_m$.
If $A$ is even, then $r_m = \t{r_m}$, so we are done.
Otherwise there is an $l$-element subset $X\subset [m]$ such that $A_{X,X}=1$.
Without loss of generality $X=\{m-l+1,m-l+2,\ldots,m\}$.
Then by Assertion~\ref{p:step-alg-l}.a'
\[
    r_m = \rk A \ge \rk D + 1 \ge r_{m - l + 1} + 1, \quad\text{where}
\]

$\bullet$ $D$ is the matrix defined in Assertion~\ref{p:step-alg-l}.a';

$\bullet$ the first inequality follows from Assertion~\ref{p:step-alg-l}.a';

$\bullet$ the second inequality holds because $D$ is an $[m-l+1]\choose l$-matrix by Assertion~\ref{p:step-alg-l}.a'.

(b)
Take an even $[m]\choose l$-matrix $A$ such that $\rk A=\t{r_m}$.
By the non-triviality $A\ne0$.
Let $X, Y, B$ be defined as in Assertion~\ref{p:step-alg-l}.b'.
Then
\[
    \t{r_m}=\rk A\ge\rk B+2 \ge r_{m-l}+2,\quad\text{where}
\]

$\bullet$ the first inequality follows from Assertion~\ref{p:step-alg-l}.b';

$\bullet$ the second inequality holds because $B$ is an even $[m]-X\choose l$-matrix by Assertion~\ref{p:hered-l}.

\smallskip
\textbf{\ref{p:rankher}.}
(a) It suffices to check that for each pair $\{\alpha, \beta\}$  the number of sets $T_{i,\{\sigma, \tau\}}$ containing $\{\alpha, \beta\}$ is odd if and only if $\alpha\cap\beta=\emptyset$ (hence this parity not depend on $i$). 
Clearly, $|\alpha \cap \beta|\le 2$.

Assume $|\alpha \cap \beta|=2$. 
Then $\{\alpha, \beta \}\notin T_{i,\{\sigma, \tau\}}$ for all $i, \sigma, \tau$ and hence $\{\alpha, \beta\}\notin A_i$ for all $i\in[5]$. 

Assume $|\alpha \cap \beta|=1$.
It suffices to consider the case $\alpha=\{1, 2\}$, $\beta=\{1, 3\}$. 
Then $\{\alpha, \beta \}\in T_{i,\{\sigma, \tau\}}$ iff $i=1$ and $\{\sigma, \tau\}$ is either $\bigl\{\{2, 4\}, \{3, 5\}\bigr\}$~or~$\bigl\{\{2, 5\}, \{3, 4\}\bigr\}$
Therefore $\{\alpha, \beta\}\notin A_i$ for all $i\in[5]$. 

Assume $|\alpha \cap \beta|=0$.
It suffices to consider the case $\alpha=\{1, 2\}$, $\beta=\{3, 4\}$.
Then $\{\alpha, \beta \}\in T_{i,\{\sigma, \tau\}}$ iff either

$\bullet$ $i=1$ and $\{\sigma, \tau\}=\bigl\{\{1, 2, 5\}, \{1, 3, 4\}\bigr\}$, or 

$\bullet$ $i=2$ and $\{\sigma, \tau\}=\bigl\{\{1, 2, 5\}, \{2, 3, 4\}\bigr\}$, or

$\bullet$ $i=3$ and $\{\sigma, \tau\}=\bigl\{\{1, 2, 3\}, \{3, 4, 5\}\bigr\}$, or 

$\bullet$ $i=4$ and $\{\sigma, \tau\}=\bigl\{\{1, 2, 4\}, \{3, 4, 5\}\bigr\}$, or 

$\bullet$ $i=5$ and $\{\sigma, \tau\}=\bigl\{\{1, 2, 5\}, \{3, 4, 5\}\bigr\}$. 

Therefore $\{\alpha, \beta\}\in A_i$ for every $i\in[5]$.

(b)
It suffices to prove that $A_{G\sqcup i,j}=A_{G\sqcup j,i}$ for each $i,j\in[m]$ and $(2l-3)$-element subset $G\subset [m]-i-j$.
Denote $\overline\sigma := \{i, j\}\sqcup\sigma$.
Then
\begin{multline*}
    A_{G\sqcup j,i}+A_{G\sqcup i,j} \overset{(1)}=
    \sum\limits_{\{(\sigma,\tau)\ :\ G=\sigma\sqcup \tau,\ |\sigma|=l-2\}}
    \left(A_{\overline\sigma, i\sqcup\tau} + A_{\overline\sigma, j\sqcup\tau}\right) \overset{(2)} = \\ =
    \sum\limits_{\{(\sigma,\tau)\ :\ G=\sigma\sqcup \tau,\ |\sigma|=l-2\}}\ \sum\limits_{t \in \tau} A_{\overline\sigma, \overline{\tau - t}}\overset{(3)}=
    \sum\limits_{t\in G}\ \sum\limits_{\{(\sigma,\nu)\ :\ G-t=\sigma\sqcup\nu,\ |\sigma|=l-2\}}
    A_{\overline\sigma, \overline\nu}
    \overset{(4)}=0,\quad\text{where}
\end{multline*}

$\bullet$ equality (1) holds because $A_{G\sqcup j,i}$ is equal to the sum of the first summands $A_{\overline\sigma,i\sqcup\tau}$,
and $A_{G\sqcup i,j}$ is equal to the sum of the second summands $A_{\overline\sigma,j\sqcup\tau}$;

$\bullet$ equality (2) holds by the linear dependence for $F=\overline\tau$, $P=\overline\sigma$;

$\bullet$ equality (3) is obtained by changes of the order of summation and of variable $\nu=\tau-t$. 

$\bullet$ equality (4) holds because ordered decompositions $(\sigma,\nu)$ of $G-t$ into $(l-2)$-element subsets $\sigma,\nu$ split into pairs $\{(\sigma,\nu),(\nu,\sigma)\}$, and 
$A_{\overline\sigma,\overline\nu}+A_{\overline\nu,\overline\sigma}=0$.

\section{Classification of symmetric bilinear forms}\label{s:bil}

This section 
illustrates the method of \S\ref{s:comhyp} and \S\ref{s:comhypge} by  simple basic examples.  

Fix a symmetric matrix $A\in\Z_2^{n\times n}$.
For $U,V\in\Z_2^n$ let
$$A(U,V)=U\cdot_A V:=\sum_{i,j=1}^n A_{i,j}U_iV_j\quad (=U^TAV).$$
A {\bf basis} of $\Z_2^n$ is an inclusion-minimal ordered set of vectors such that every vector from $\Z_2^n$ is the sum of some vectors from this set. 

\begin{theorem}\label{t:2dim} 
For $n=2$ there is a basis $X_1,X_2$ of $\Z_2^2$ and numbers $\gamma_1,\gamma_2\in\Z_2$ such that either

(i) for any $a_1,a_2,b_1,b_2\in\Z_2$ we have
$$(a_1X_1+a_2X_2)\cdot_A(b_1X_1+b_2X_2)=\gamma_1a_1b_1+\gamma_2a_2b_2,\quad\text{or}$$

(ii) for any $a_1,a_2,b_1,b_2\in\Z_2$ we have
$$(a_1X_1+a_2X_2)\cdot_A(b_1X_1+b_2X_2) = a_1b_2+a_2b_1.$$
\end{theorem}

Recall that assertions and problems stated after theorems are hints to proofs of the theorems. 

\begin{problem}\label{pm:odd:2dim} Assume that $n=2$, $X\in\Z_2^2$ and $X\cdot_A X=1$.

(a) For any $P\in\Z_2^2$ there is $\lambda_{X,P}\in\Z_2$ such that
for $P_X:=P+\lambda_{X,P}X$ we have $P_X\cdot_A X=0$.


(b) There is a basis $X_1=X$, $X_2$ of $\Z_2^2$ and numbers
$\gamma_1=1,\gamma_2\in\Z_2$ such that the property (i) of Theorem \ref{t:2dim} holds. 
\end{problem}

\begin{problem}\label{pm:even:2dim}
Assume that $n=2$, $X,Y\in\Z_2^2$ and $X\cdot_A Y=1$, $X\cdot_A X=Y\cdot_A Y=0$.
Then $X_1:=X$, $Y_1:=Y$ is a basis of $\Z_2^2$ such that the property (ii) of Theorem \ref{t:2dim} holds. 
\end{problem}

\begin{theorem}\label{t:3dim} For $n=3$ there is a basis $X_1,X_2,X_3$ of $\Z_2^3$ and numbers $\gamma_1,\gamma_2,\gamma_3\in\Z_2$ such that either

(i) for any $a_1,a_2,a_3,b_1,b_2,b_3\in\Z_2$ we have
$$(a_1X_1+a_2X_2+a_3X_3)\cdot_A(b_1X_1+b_2X_2+b_3X_3) = \gamma_1a_1b_1+\gamma_2a_2b_2+\gamma_3a_3b_3,\quad\text{or}$$

(ii) for any $a_1,a_2,a_3,b_1,b_2,b_3\in\Z_2$ we have
$$(a_1X_1+a_2X_2+a_3X_3)\cdot_A(b_1X_1+b_2X_2+b_3X_3) = a_1b_2+a_2b_1+\gamma_3a_3b_3.$$
\end{theorem}

\begin{problem}\label{pm:even:3dim}
Assume that $X,Y\in\Z_2^3$ and $X\cdot_A Y=1$, $X\cdot_A X=Y\cdot_A Y=0$.
 
(a) For any $P\in\Z_2^3$ there are $\lambda_{X,Y,P},\lambda_{Y,X,P}\in\Z_2$ such that for $P_{X,Y}:=P+\lambda_{X,Y,P}Y+\lambda_{Y,X,P}X$ we have $P_{X,Y}\cdot_A X=P_{X,Y}\cdot_A Y=0$.


(b) There is a basis $X_1=X$, $X_2=Y$, $X_3$ of $\Z_2^3$ and a number $\gamma_3\in\Z_2$ such that the property (ii) of Theorem \ref{t:3dim} holds. 
\end{problem}

\begin{theorem}\label{t:classif} There are $k,l$ and a basis $X_1,Y_1,\ldots,X_k,Y_k,Z_1,\ldots,Z_{n-2k}$ of $\Z_2^n$ such that $2k+l\le n$ and for any
$a,a',b,b'\in\Z_2^k$ and $c,c'\in\Z_2^{n-2k}$ we have
$$(a_1X_1+b_1Y_1+\ldots+a_kX_k+b_kY_k+c_1Z_1+\ldots+c_{n-2k}Z_{n-2k}) \cdot_A $$ $$\cdot_A(a_1'X_1+b_1'Y_1+\ldots+a_k'X_k+b_k'Y_k+c_1'Z_1+\ldots+c_{n-2k}'Z_{n-2k}) =$$ $$=a_1b_1'+a_1'b_1+\ldots+a_kb_k'+a_k'b_k+c_1c_1'+\ldots+c_lc_l'.$$
If $A$ is even, then $l=0$.
\end{theorem}


\begin{pr}\label{pm:odd} Assume that $X\in\Z_2^n$ and $X\cdot_A X=1$.

(a) State and prove the $n$-dimensional analogue of Assertion \ref{pm:odd:2dim}.a.

\emph{Hint: $\lambda_{X,P}=X\cdot_A P$.}


(b) There is a basis $X,E_1,\ldots,E_{n-1}$ of $\Z_2^n$ and a symmetric matrix
$B\in\Z_2^{(n-1)\times(n-1)}$ such that for any $a,b\in\Z_2$ and $\lambda,\mu\in \Z_2^{n-1}$ we have $$(aX+\lambda_1E_1+\ldots+\lambda_{n-1}E_{n-1})\cdot_A(bX+\mu_1E_1+\ldots+\mu_{n-1}E_{n-1})=ab+\lambda\cdot_B\mu.$$
\end{pr}

\begin{pr}\label{pm:even} Assume that $X,Y\in\Z_2^n$ and $X\cdot_A Y=1$, $X\cdot_A X=Y\cdot_A Y=0$.

(a) State and prove the $n$-dimensional analogue of Assertion \ref{pm:even:3dim}.a.

\emph{Hint: $\lambda_{X,Y,P}=X\cdot_A P$, $\lambda_{Y,X,P}=Y\cdot_A P$.}

(b) There is a basis $X,Y,E_1,\ldots,E_{n-2}$ of $\Z_2^n$ and a symmetric matrix
$B\in\Z_2^{(n-2)\times(n-2)}$ such that for any $a_X,a_Y,b_X,b_Y\in\Z_2$ and $\lambda,\mu\in \Z_2^{n-2}$ we have $$(a_XX+a_YY+\lambda_1E_1+\ldots+\lambda_{n-2}E_{n-2})\cdot_A(b_XX+b_YY+\mu_1E_1+\ldots+\mu_{n-2}E_{n-2}) = a_Xb_Y+a_Yb_X+\lambda\cdot_B\mu.$$
\end{pr}


{\it Books, surveys, and expository papers in this list are marked by the stars.}

\end{document}